\documentclass[11pt]{amsart}
\usepackage{amsmath,amsthm,latexsym,amssymb}
\usepackage{graphicx}
\usepackage{color}
\usepackage{epsfig}
\numberwithin{equation}{section}

\newtheorem{theorem}{Theorem}
\newtheorem{lemma}{Lemma}

\newtheorem{proposition}{Proposition}

\numberwithin{theorem}{section}

\numberwithin{corollary}{section}
\numberwithin{lemma}{section}
\numberwithin{definition}{section}
\numberwithin{proposition}{section}
\numberwithin{remark}{section}

\newcommand{\R}{\mathbb R}

\newcommand{\medint}{-\kern  -,375cm\int}

\title[A remark on optimal weighted Poincar\'e inequalities]{A remark on optimal weighted Poincar\'e inequalities for convex domains}
\author[V. Ferone - C. Nitsch  - C. Trombetti]{V. Ferone$^*$ - C. Nitsch$^*$    - C. Trombetti$^*$}
 \thanks{%
$^*$ Dipartimento di Matematica e Applicazioni ``R. Caccioppoli'', Universit\`{a}
degli Studi di Napoli Federico II, Complesso Monte S. Angelo, via Cintia
- 80126 Napoli, Italy; \\ email: \tt ferone@unina.it; c.nitsch@unina.it;
cristina@unina.it}


\begin{document}

\maketitle

\begin{abstract}
We prove a  sharp upper bound on convex domains, in terms of the diameter alone, of the best constant in a class of weighted Poincar\'e inequalities. The key point is the study of an optimal weighted Wirtinger inequality.
\end{abstract}

\section{Introduction}
%

It is well known that (see for instance \cite{Ma}), for any given open bounded Lipschitz connected set $\Omega$, a Poincar\'e inequality holds true, in the sense that there exists a positive constant $C_{\Omega,p}$ such that
\begin{equation}\label{eq_poinc}
\inf_{t\in \R} \|u-t\|_{L^p(\Omega)}\le C_{\Omega,p} \|Du\|_{L^p(\Omega)},
\end{equation}
for all Lipschitz functions $u$ in $\Omega$.

The value of the best constant in \eqref{eq_poinc} is the reciprocal of the first nontrivial Neumann eigenvalue of the $p$-Laplacian over $\Omega$.
 In \cite{PW} (see also \cite{B}), it has been proved that, if $p=2$, and $\Omega$ is convex, in any dimension
\begin{equation}\label{eq_classica}
\frac1{C_{\Omega,2}}=\mathop{\min_{u\in H^{1}(\Omega)}}_{\int_\Omega u=0} \dfrac{\displaystyle\left(\int_\Omega |Du|^2\right)^\frac12}{\displaystyle\left(\int_\Omega |u|^2\right)^\frac12}\ge\frac{\pi}{d},
\end{equation}
where $d$ is the diameter of $\Omega$.
Observe that the last term of \eqref{eq_classica} is exactly the value achieved,  in dimension $n=1$ on any interval of length $d$, by the first nontrivial Laplacian eigenvalue (without distinction between the Neumann and the Dirichlet conditions).

The proof of \eqref{eq_classica} in \cite{PW, B} indeed relies on the reduction to a one dimensional problem. At this aim, for any given smooth admissible test function  $u$ in the Rayleigh quotient in \eqref{eq_classica}, the authors show that it is possible to perform a clever slicing of the domain $\Omega$ in convex sets which are as tiny as desired in at least $n-1$ orthogonal directions. On each one of such convex components of $\Omega$, they are able to show that the Rayleigh quotient of $u$ can be approximated by a 1-dimensional weighted Rayleigh quotient. This leads the authors to look for the best constants of a class of one dimensional weighted Poincar\'e-Wirtinger inequalities.
The result was later generalized to $p=1$ in \cite{AD} and only recently to $p\ge 2$ in \cite{ENT} and for any $p>1$ in the framework of compact manifolds in \cite{NV,V}.
Other optimal Poincar\'e inequalities can be found in \cite{BCT, BT, BCM, CDB, EFKNT, S, W}.

In this paper we consider a weighted Poincar\'e inequality, namely, for $p>1$, given a nonnegative log-concave function $\omega$ on $\Omega$ there exists a positive constant $C_{\Omega,p,\omega}$ such that, for every Lipschitz function $u$
\begin{equation}\label{eq_wepoinc}
\inf_{t\in \R} \|u-t\|_{L_\omega^p(\Omega)}\le C_{\Omega,p,\omega} \|Du\|_{L_\omega^p(\Omega)}.
\end{equation}
Here $\|\cdot\|_{L_\omega^p}$ denotes the weighted Lebesgue norm. 
The best constant $C_{\Omega,p,\omega}$ in \eqref{eq_wepoinc} is given by 
\begin{equation}\label{eq_quoz}
\frac1{C_{\Omega,p,\omega}}=\mathop{\inf_{u \,\text{Lipschitz}}}_{\int_\Omega |u|^{p-2}u\,\omega=0} \dfrac{\displaystyle\left(\int_\Omega |Du|^p\,\omega\right)^\frac1p}{\displaystyle\left(\int_\Omega |u|^p\,\omega\right)^\frac1p}
\end{equation}
Our main result is the following.
\begin{theorem}[Main Theorem]\label{teo_main}
Let $\Omega\subset\R^n$ be a bounded convex set having diameter $d$ and let $\omega$ be a nonnegative log-concave function on $\Omega$. For $p> 1$ and in any dimension we have 
\begin{equation}\label{eq_main}
C_{\Omega,p,\omega} \le  \frac{d}{\pi_p}
\end{equation}
where 
\begin{equation}\label{eq_pp}
\pi_p=2\int_0^{+\infty} \frac{1}{1+\frac{1}{p-1}s^p}ds=2\pi\frac{(p-1)^{1/p}}{p(\sin(\pi/p))}.
\end{equation}
\end{theorem}

Other optimal weighted inequalities can be found in literature (see \cite{Ma} and the reference therein) but to our knowledge similar explicit sharp bounds were obtained only for $p=1$ and $p=2$ (see for instance \cite{CW1,CW2}). 
We observe that when $\omega=1$ estimate \eqref{eq_main} is the optimal estimate already obtained in the unweigthed case in \cite{AD,B,ENT,PW,V}.    
Indeed $\dfrac{d}{\pi_p}$ is the optimal constant of the one-dimensional unweighted Poincar\'e--Wirtinger inequality on a segment of length $d$ (see for instance \cite{BBCDG}, \cite{ReWa}), namely:

$$\frac{\pi_p}{d}=\displaystyle\inf_{W_0^{1,p}(0,d)}\displaystyle\frac{\displaystyle\left(\int_0^d|u'|^p\right)^\frac1p}{\displaystyle\left(\int_0^d|u|^p\right)^\frac1p}$$

Explicit expression for $\pi_p$ (as \eqref{eq_pp}) can be found for instance in \cite{ BBCDG, Li, ReWa, sta1, sta2}. The fact that $\pi_2=\pi$ is consistent with the classical Wirtinger inequality (see \cite{HL}) and obviously also with \eqref{eq_classica}.


In the spirit of the proof of \eqref{eq_classica} by Payne and Weinberger \cite{PW}, our proof of Theorem \ref{teo_main} is based on the following estimate on the best constant in a class of weighted Wirtinger inequalities. 

\begin{proposition}\label{pr_wirtinger}
Let $f$ be a nonnegative log-concave function defined on $[0,L]$ and $p> 1$
then
\begin{equation}\label{eq_wirtinger}
\mathop{\inf_{u\in W^{1,p}(0,L)}}_{\int_0^L |u|^{p-2}u\,f=0} \frac{\displaystyle \int_0^L |u'(x)|^p f(x)dx}{\displaystyle\int_0^L |u(x)|^p f(x)dx}\ge\mathop{\min_{u\in W^{1,p}(0,L)}}_{\int_0^L |u|^{p-2}u=0} \frac{\displaystyle \int_0^L |u'(x)|^p dx}{\displaystyle\int_0^L |u(x)|^p dx}=\left(\frac{\pi_p}{L}\right)^p.
\end{equation}
\end{proposition}
For an insight into generalized Wirtinger inequalities, and more generally into weighted Hardy inequalities, we refer to \cite{KP,Ma,Tu}, other results can be found for instance in  \cite{CD, DGS, FGR, GR, R}.

In Section \ref{sec_wirtinger}, we prove Proposition \ref{pr_wirtinger}, while in Section \ref{sec_red}, we employ a ``slicing argument'' to pass from the $n$--dimensional to the one--dimensional case.

\section{Proof of Proposition \ref{pr_wirtinger}}\label{sec_wirtinger}

For the reader convenience we have split the claim in two Lemmata.

In \cite{ENT} it has been proved the following lemma for which we include the proof for the sake of completeness.
\begin{lemma}\label{lem_kappa}
Let $f$ be a smooth positive log-concave function defined on $[0,L]$ and $p>1$.
Then there exists $\kappa\in\R$ such that 
\begin{equation}\label{eq_kappa}
\mathop{\inf_{u\in W^{1,p}(0,L)}}_{\int_0^L |u|^{p-2}u\,f=0} \frac{\displaystyle \int_0^L |u'(x)|^p f(x)dx}{\displaystyle\int_0^L |u(x)|^p f(x)dx}\ge \mathop{\inf_{u\in W^{1,p}(0,L)}}_{\int_0^L |u|^{p-2}u\,e^{\kappa x}=0} \frac{\displaystyle \int_0^L |u'(x)|^p e^{\kappa x}dx}{\displaystyle\int_0^L |u(x)|^p e^{\kappa x}dx}
\end{equation}
\end{lemma}
\begin{proof}
By standard compactness argument (see \cite[Theorem 1.5 pag. 28]{KP}) the positive infimum on the left had side of \eqref{eq_kappa} is achieved by some function $u_\lambda$
belonging to $$\left\{u\in W^{1,p}(0,L),\int_0^L |u(x)|^{p-2}u(x)f(x)dx=0\right\}.$$ As aspected, such a minimizer is also a $C^1(0,L)$ solution to the following Neumann eigenvalue problem 
\begin{equation}\label{eq_gangbo}
\left\{\begin{array}{ll}
(-u'|u'|^{p-2})'=\lambda u|u|^{p-2}+h'(x)u'|u'|^{p-2} & x\in(0,L)\\\\
u'(0)=u'(L)=0.
\end{array}\right.
\end{equation} 
Here $h(x)=\log f(x)$ is a smooth bounded concave function and $\lambda$ is  the left hand side of  \eqref{eq_kappa}.
We emphasize that the usual derivation of \eqref{eq_gangbo} as the Euler Lagrange of a Rayleigh quotient it is rigorous only to handle the case $p\ge 2$. A refined technique similar to the one worked out in \cite[Lemma 2.4]{DGS} is necessary when $1<p<2$.  

Since it is not difficult to prove that for all $0<L_1<L$
\begin{equation*}
\mathop{\inf_{u\in W^{1,p}(0,L)}}_{\int_0^L |u|^{p-2}u\,f=0} \frac{\displaystyle \int_0^L |u'(x)|^p f(x)dx}{\displaystyle\int_0^L |u(x)|^p f(x)dx}< \mathop{\inf_{u\in W^{1,p}(0,L_1)}}_{\int_0^{L_1} |u|^{p-2}u\,f=0} \frac{\displaystyle \int_0^{L_1} |u'(x)|^p f(x)dx}{\displaystyle\int_0^{L_1} |u(x)|^p f(x)dx},
\end{equation*}
then $u_\lambda$ vanishes in one and only one point namely $x_\lambda\in(0,L)$  and without loss of generality we may assume that ${u_\lambda}(L)<0<{u_\lambda}(0)$.

We claim that if $\kappa =h'(x_\lambda)$ then 

$$\lambda\ge \mathop{\min_{u\in W^{1,p}(0,L)}}_{\int_0^L |u|^{p-2}u\,e^{\kappa x}=0} \frac{\displaystyle \int_0^L |u'(x)|^pe^{\kappa x} dx}{\displaystyle\int_0^L |u(x)|^p e^{\kappa x}dx} \equiv\bar\lambda.$$

Arguing by contradiction we assume that $\lambda< \bar \lambda$. Therefore there exists a function $u_{\bar \lambda}$ solution to
\begin{equation*}
\left\{\begin{array}{ll}
(-u'|u'|^{p-2})'=\bar\lambda  u|u|^{p-2}+h'(x_\lambda) u'|u'|^{p-2} & x\in(0,L)\\\\
 u'(0)= u'(L)=0
\end{array}\right.
\end{equation*} 

Standard arguments ensures that $u_{\bar\lambda}$ is strictly monotone in $(0,L)$ and therefore vanishes in one and only one point namely $x_{\bar \lambda}\in(0,L)$. We assume without loss of generality that $u_{\bar\lambda}(L) < 0 < u_{\bar \lambda}(0)$. Since $h'$ is non increasing in $[0,L]$, a strightforward consequence of the comparison principle applied to $u_\lambda$ and $u_{\bar\lambda}$ on the interval $[0,x_\lambda]$ enforces $x_{\bar \lambda}< x_\lambda$. On the other hand the comparison principle applied to $u_\lambda$ and $u_{\bar\lambda}$ on the interval $[x_\lambda,L]$ enforces $x_{\bar \lambda}> x_\lambda$ and eventually a contradiction arises.
\end{proof}



\begin{lemma}\label{lem_allkappa}
For all $\kappa\in\R$ and $p>1$
\begin{equation}\label{eq_ND}
\mathop{\min_{u\in W^{1,p}(0,L)}}_{\int_0^L |u|^{p-2}u\,e^{\kappa x}=0} \frac{\displaystyle \int_0^L |u'(x)|^pe^{\kappa x} dx}{\displaystyle\int_0^L |u(x)|^p e^{\kappa x}dx}\ge \left(\frac{\pi_p}{L}\right)^p
\end{equation}
\end{lemma}
\begin{proof}
If $u$ minimizes the left hand side of \eqref{eq_ND} then it solves
\begin{equation*}
\left\{\begin{array}{ll}
(-u'|u'|^{p-2})'=\mu u|u|^{p-2}+\kappa u'|u'|^{p-2} & x\in(0,L)\\\\
u'(0)=u'(L)=0,
\end{array}\right.
\end{equation*} 
where $$\mu=\mathop{\min_{u\in W^{1,p}(0,L)}}_{\int_0^L |u|^{p-2}u\,e^{\kappa x}=0} \frac{\displaystyle \int_0^L |u'(x)|^p e^{\kappa x}dx}{\displaystyle\int_0^L |u(x)|^p e^{\kappa x}dx}.$$
As in the previous Lemma $u'$ can not vanish inside $(0,L)$ and 
we may assume without loss of generality that $u$ is an increasing function such that $u(0)<0<u(L)$. Then the function $v(x)=\displaystyle\frac{u(x)}{u'(x)}$ 
is the increasing solution to the following problem
\begin{equation}\label{eq_v} 
\left\{ \begin{array}{ll}
v'=1+\frac{1}{p-1}\left(\mu |v|^{p}+\kappa v\right)& x\in(0,L)\\\\
\displaystyle\lim_{x\to L} v(x)=-\lim_{x\to 0} v(x)=+\infty,
\end{array}\right.
\end{equation}
where uniqueness and monotonicity come easily form the fact that the equation in \eqref{eq_v} is autonomous.


In particular we observe that $\mu |y|^{p}+\kappa y + p-1=0$ can not have solutions $y\in \R$. The fact that $v'$ is bounded away from zero allows us to integrate $\dfrac{1}{v'}$ with respect to $v$ obtaining

\begin{align*}
L&=\int_{-\infty}^{+\infty}\frac{1}{v'}dv=\int_{-\infty}^{+\infty}\frac{1}{1+\frac{1}{p-1}\left(\mu|v|^p-\kappa v \right)}dv\\
&=\int_{0}^{+\infty}\left(\frac{1}{1+\frac{1}{p-1}\left(\mu v^p+\kappa v \right)}+\frac{1}{1+\frac{1}{p-1}\left(\mu v ^p-\kappa v \right)} \right)dv\\
&\ge 2 \int_{0}^{+\infty}\frac{1}{1+\frac{1}{p-1} \mu v ^p}dv,\\
\end{align*}
and the proof is complete observing that rescaling $s=\dfrac{\pi_p}{L}v$ in \eqref{eq_pp} gives
$$L = 2 \int_{0}^{+\infty}\frac{1}{1+\frac{1}{p-1} \left(\frac{\pi_p}{L}\right)^pv ^p}dv.$$
\end{proof}
When the function $f$ is smooth, log-concave and positive, Proposition \ref{pr_wirtinger} is a consequence of Lemma \ref{lem_kappa} and Lemma \ref{lem_allkappa}. In the general case Proposition \ref{pr_wirtinger} follows by approximation arguments.

\section{Proof of Theorem \ref{teo_main}}\label{sec_red}
The aim of this section is to prove that Theorem \ref{teo_main} can be deduced from Proposition \ref{pr_wirtinger}. As we already mentioned the idea is based on a slicing method worked out in \cite{PW} and proved in a slightly different way also in \cite{AD,B,CW1}. We outline the technique for the sake of completeness.

\begin{lemma}
\label{division}
Let $\Omega$ be a convex set in $\R^n$ having diameter $d$, let $\omega$ be a nonnegative log-concave function on $\Omega$, and let $u$ be any function such that $\int_{\Omega} |u(x)|^{p-2}u(x)\omega(x)dx=0$. Then, for all positive $\varepsilon$, there exists a decomposition of the set $\Omega$ in mutually disjoint convex sets $\Omega_i$ ($i=1,...,k$) such that
$$\bigcup_{i=1}^k\bar\Omega_i=\bar\Omega$$
$$\int_{\Omega_i}|u(x)|^{p-2}u(x)\omega(x)dx=0$$
and for each $i$ there exists a rectangular system of coordinates such that
$$\Omega_i\subseteq\{(x_1,...,x_n)\in\R^n: 0\le x_1\le d_i, |x_\ell|\le\varepsilon, \ell=2,...,n \}\quad (d_i\le d, i=1,...,k)$$
\end{lemma}
\begin{proof}
Among all the $n-1$ hyperplanes of the form $a x_1+ b x_2 = c$, orthogonal to the 2-plane $\Pi_{1,2}$ generated by the directions $x_1$ and $x_2$, by continuity there exists certainly one that divides $\Omega$ into two nonempty subsets on each of which the integral of $ u|u|^{p-2}\omega$ is zero and their projections on $\Pi_{1,2}$ have the same area. We go on subdividing recursively in the same way both subset and eventually we stop when all the subdomains $\Omega^{(1)}_j$ ($j=1,...,2^{N_1}$) have projections with area smaller then $\varepsilon^2/2.$ Since the width $w$ of a planar set of area $A$ is bounded by the trivial inequality $w\le\sqrt{2A}$, each subdomain $\Omega^{(1)}_j$ can be bounded by two parallel $n-1$ hyperplanes of the form $a x_1+ b x_2 = c$ whose distance is less than $\varepsilon$. If $n=2$ the proof is completed, provided that we understand $\Pi_{1,2}$ as $\R^2$, the projection of $\Omega$ on $\Pi_{1,2}$ as $\Omega$ itself, and the $n-1$ orthogonal hyperplanes as lines.
If $n>2$ for any given $\Omega^{(1)}_j$ we can consider a rectangular system of coordinates such that the normal to the above $n-1$ hyperplanes which bound the set, points in the direction $x_n$. Then we can repeat the previous arguments and subdivide the set $\Omega^{(1)}_i$ in subsets $\Omega^{(2)}_j$ ($j=1,...,2^{N_2}$) on each of which the integral of $u|u|^{p-2}\omega$ is zero and their projections on $\Pi_{1,2}$ have the same area which is less then $\varepsilon^2/2$. Therefore, any given $\Omega^{(2)}_j$, can be bounded by two $n-1$ hyperplanes of the form $a x_1+ b x_2 = c$ whose distance is less than $\varepsilon$. If $n=3$ the proof is over. If $n>3$ we can go on considering $\Omega^{(2)}_j$ and rotating the coordinate system such that the normal to the above $n-1$ hyperplanes which bound $\Omega^{(2)}_j$, points in the direction $x_{n-1}$ and such that the rotation keeps the $x_n$ direction unchanged. The procedure ends after $n-1$ iterations, at that point we have perfomed $n-1$ rotations of the coordinate system and all the directions have been fixed. Up to a translation, in the resulting coordinate system $$\Omega^{(n-1)}_j\subseteq\{(x_1,...,x_n)\in\R^n: 0\le x_1\le d_j, |x_\ell|\le\varepsilon, \ell=2,...,n \}$$
\end{proof}

\begin{proof}[Proof of Theorem \ref{teo_main}]
From \eqref{eq_quoz}, using the density of smooth functions in Sobolev spaces it will be enough to prove that 
$$ \dfrac{\displaystyle\int_\Omega |Du|^p\omega }{\displaystyle\int_\Omega |u|^p\omega}
\ge \left(\frac{\pi_p}{d}\right)^p $$
when $u$ is a smooth function with uniformly continuous first derivatives and $\int_{\Omega}|u(x)|^{p-2}u(x)\omega(x)dx=0$.

Let $u$ be any such function. According to Lemma \ref{division} we fix $\varepsilon>0$ and we decompose the set $\Omega$ in convex domains $\Omega_i$ ($i=1,...,k$). We use the notation of Lemma \ref{division} and we focus on one of the subdomains $\Omega_i$ and 
fix the reference system such that 
$$\Omega_i\subseteq\{(x_1,...,x_n)\in\R^n: 0\le x_1\le d_i, |x_\ell|\le\varepsilon, \ell=2,...,n \}.$$
For $t\in[0,d_i]$ we denote by $g_i(t)$ the $n-1$ volume of the intersection of $\Omega_i$ with the $n-1$ hyperplane $x_1=t$. Since $\Omega_i$  is convex, then by Brunn-Minkowski inequality (see \cite{G}) $g_i$ is a log-concave function in $[0,d_i]$.

Then, for any $t\in[0,d_i]$ we denote by $v(t)=u(t,0,...,0)$, and $f_i(t)=g_i(t)\omega(t,0,\dots,0)$.  Since $u$, $\dfrac{\partial u}{\partial {x_1}}$ and $\omega$ are uniformly continuous in $\Omega$, there exists a modulus of continuity $\eta(\cdot)$ ($\eta(\varepsilon)\searrow 0$ as $\varepsilon\to 0$) independent of the decomposition of $\Omega$ such that

\begin{equation}
\label{uno}
\left | \int_{\Omega_{i}} \left|\frac{\partial u}{\partial {x_1}} \right|^p \omega \> dx-\int_0^{d_i} |v'(t)|^p f_i(t) \> dt\right| \leq \eta(\varepsilon)|\Omega_i| ,
\end{equation}

\begin{equation}
\label{due}
\left | \int_{\Omega_{i}} \left|u \right|^p\omega \> dx-\int_0^{d_i} |v(t)|^p f_i(t) \> dt\right| \leq \eta(\varepsilon)|\Omega_i| 
\end{equation}

and

\begin{equation}
\label{tre}
\left | \int_0^{d_i} |v(t)|^{p-2}v(t)f_i(t) \> dt\right| \leq \eta(\varepsilon)|\Omega_i| .
\end{equation}

Since $d_i\le d$, and $f_i$ are nonnegative log-concave functions, applying Proposition \ref{pr_wirtinger} we have 

$$
\int_{\Omega_i} | D u|^p\omega \>dx\ge  \int_{\Omega_{i}} \left|\frac{\partial u}{\partial {x_1}} \right|^p\omega \> dx \ge
\left(\frac{\pi_p}{d}\right)^p \int_{\Omega_{i}} \left|u(x) \right|^p \omega\> dx + C\eta(\varepsilon)  |\Omega_i|.
$$
Here the constant $C$ does not depend on $\varepsilon$.
Summing up the last inequality over all $i=1,...,k$

$$
\int_{\Omega} | D u|^p \omega\>dx\ge \left(\frac{\pi_p}{d}\right)^p \int_{\Omega} \left|u(x) \right|^p\omega \> dx + C \eta( \varepsilon)  |\Omega|.
$$

and as $\varepsilon\to 0$ we obtain the desired inequality.

\end{proof}

\end{document}